\documentclass[a4paper,11pt]{amsart}
\usepackage{fancyhdr}
\usepackage{amsmath}
\usepackage{dsfont}
\usepackage{hyperref}
\usepackage[mathscr]{eucal}
\usepackage[cp1251]{inputenc}
\usepackage[english]{babel}
\usepackage{enumerate,float,indentfirst}
\usepackage{graphicx}
\usepackage{xcolor}
\usepackage{latexsym,a4,mathrsfs,amsthm,amsmath,amssymb,url}
\usepackage{amsfonts}
\usepackage{amssymb}
\usepackage{caption}

\usepackage{stackengine}

\numberwithin{equation}{section}
\newtheorem{lemma}{Lemma}

\newtheorem{corollary}{Corollary}
\newtheorem{proposition}{Proposition}

\begin{document}

\vspace{0in}

\title[Submatrices with the best-bounded inverses]{\bf Submatrices with the best-bounded inverses: \\ an asymptotically tight upper bound for $\mathbb{C}^{n \times 2}$ }

\author[Yu. Nesterenko]{Yuri Nesterenko}
\email{yuri.r.nesterenko@gmail.com}

\begin{abstract}
The long-standing hypothesis formulated by Goreinov, Tyrtyshnikov and Zamarashkin \cite{GTZ1997} has recently been solved affirmatively in the case of real two-column matrices by Sengupta and Pautov \cite{SP2026}. 
In this paper, we consider the complex variant of this problem and prove the asymptotically tight upper bound for spectral norms of the
best-bounded inverse $2 \times 2$ submatrices of an arbitrary complex $n \times 2$ matrix with orthonormal columns.
\end{abstract}

\maketitle

\thispagestyle{empty}
\vspace{-5truemm}
\section{Introduction}

The long-standing hypothesis formulated by Goreinov, Tyrtyshnikov and Zamarashkin \cite{GTZ1997} has recently been solved affirmatively in the case of real two-column matrices by Sengupta and Pautov \cite{SP2026}. 
In this paper, we consider the complex variant of this problem and prove the following statement.

\begin{proposition}\label{pr}
For every $n \geq 3$ and an arbitrary complex $n \times 2$ matrix with orthonormal columns, there exists a $2 \times 2$ submatrix such that the spectral norm of its inverse does not exceed $\sqrt{n / \alpha}$, where $\alpha = 2-2/\sqrt3$. The bound is asymptotically tight. Moreover, equality for a fixed $n$ occurs if and only if $4 \mid n$ and the Hopf images\footnote{The Hopf map $p: \mathbb{C}^2 \rightarrow \mathbb{R}^3$ is given by the formula $p(u, v) = ( \bar{u} v + u \bar{v}, \, i(\bar{u} v - u \bar{v}), \, |u|^2 - |v|^2 )$.} of the rows of the matrix are clustered into four groups of $n / 4$ identical vectors whose endpoints taken one from each cluster form the vertices of a regular tetrahedron centered at the origin.
\end{proposition}

This yields the following equivalent statement about spatial polygons discussed in \cite{Nesterenko2024}.
\begin{corollary}\label{co}
For arbitrary vectors $w_1, \ldots, w_n, \, n \geq 3$ forming a polygon of perimeter 2 in the $3$-dimensional Euclidean space
\begin{equation*}
\begin{split}
1) \, & \max_{i \neq j} (|w_i| + |w_j| - |w_i + w_j|) \geq 2\alpha/n, \text{ where } \alpha = 2-2/\sqrt3, \\
2) \, &\text{the bound is asymptotically tight,} \\
3) \, & \text{equality holds} \iff 4 \mid n, \text{ the multiset } \{ w_1, \ldots, w_n \} \text{ consists of four} \\ 
& \text{distinct vectors, each occurring exactly } n/4 \text{ times, with endpoints} \\
& \text{forming the vertices of a regular tetrahedron centered at the origin}.
\end{split}
\end{equation*}
\end{corollary} 

The correspondence between these two formulations can be established by computing the Hopf images of the rows of the matrix (see \cite{HK1997}, \cite{Nesterenko2024}).

\section{Proof}

We will use the following lemma, whose proof was discovered with assistance from ChatGPT, model GPT-5.4 Pro.
\begin{lemma}\label{le}
Let $w_1,\ldots,w_n\in\mathbb R^3$ satisfy
\begin{equation*}
    \sum_{i=1}^n w_i=0,
    \quad
    \sum_{i=1}^n |w_i|=2.
\end{equation*}
Let $\alpha=2 - 2/\sqrt3$ and let $M=(M_{ij})_{i,j=1}^n$ be the real $n\times n$ matrix
\begin{equation*}
    M_{ij}
    =(w_i, w_j) - (|w_i|-\frac{2\alpha}{n})(|w_j|-\frac{2\alpha}{n}) + \frac{2\alpha^2}{n^2}.
\end{equation*}
Then $M$ is not entrywise positive.
\end{lemma}

\textbf{Proof.} Put
\begin{equation*}
    r_i=|w_i|,
    \quad
    \tau=\frac{2\alpha}{n}.
\end{equation*}
Also define
\begin{equation*}
    P_{ij}=r_i r_j-(w_i,w_j).
\end{equation*}
By Cauchy's inequality, $P_{ij}\ge 0$ for all $i,j$, and clearly $P_{ii} = 0$. Moreover, since $\sum_i r_i=2$ and $\sum_i w_i=0$, we have
\begin{equation}\label{Psumm}
    \sum_{j=1}^n P_{ij}
    =r_i\sum_{j=1}^n r_j-\left(w_i,\sum_{j=1}^n w_j\right)
    =2r_i.
\end{equation}

We first rewrite $M_{ij}$ in terms of $P_{ij}$,
\begin{align*}
M_{ij}=&(w_i,w_j)-(r_i-\tau)(r_j-\tau)+\frac{2\alpha^2}{n^2} = \\
&=-P_{ij}+\tau(r_i+r_j)-\frac{\tau^2}{2}.
\end{align*}
Equivalently, if
\begin{equation*}
    \rho_i=r_i-\frac{\tau}{4},
\end{equation*}
then
\begin{equation}
    M_{ij}=\tau(\rho_i+\rho_j)-P_{ij}.
    \label{eq:M-rho-P}
\end{equation}

Assume, for contradiction, that all entries of $M$ are strictly positive. Then \eqref{eq:M-rho-P} gives
\begin{equation*}
    P_{ij}<\tau(\rho_i+\rho_j)
    \qquad\text{for all }i,j.
\end{equation*}
Taking $i=j$, and using $P_{ii}=0$, we get
\begin{equation*}
    0<M_{ii}=2\tau\rho_i,
\end{equation*}
so every $\rho_i$ is positive.

Let
\begin{equation*}
    R_2=\sum_{i=1}^n r_i^2,
    \quad
    F=\sum_{i,j=1}^n P_{ij}^2.
\end{equation*}
Since all $P_{ij}\ge 0$ and there is at least one positive among them, the strict inequality above implies
\begin{align*}
F < &\tau\sum_{i,j=1}^n P_{ij}(\rho_i+\rho_j)
= 2\tau\sum_{i=1}^n \rho_i\sum_{j=1}^n P_{ij} = \\
&= 2\tau\sum_{i=1}^n \rho_i(2r_i) 
= 4\tau\sum_{i=1}^n \rho_i r_i.
\end{align*}
But
\begin{equation*}
    \sum_{i=1}^n \rho_i r_i
    =\sum_{i=1}^n r_i^2-\frac{\tau}{4}\sum_{i=1}^n r_i
    =R_2-\frac{\tau}{2}
    =R_2-\frac{\alpha}{n}.
\end{equation*}
Hence
\begin{equation}
    F<4\tau\left(R_2-\frac{\alpha}{n}\right)
    =\frac{8\alpha}{n}\left(R_2-\frac{\alpha}{n}\right).
    \label{eq:upper-bound}
\end{equation}

We now prove the opposite inequality. Let $r=(r_1,\ldots,r_n)^T$ and let
\begin{equation*}
    G=((w_i,w_j))_{i,j=1}^n.
\end{equation*}
Then
\begin{equation*}
    P = (P_{ij})_{i,j=1}^n = rr^T-G.
\end{equation*}
The matrix $G$ is positive semidefinite and has rank at most $3$, because it is the Gram matrix of vectors in $\mathbb R^3$. Since $rr^T$ is positive semidefinite, matrix $P$ has at most $3$ negative eigenvalues. Also,
\begin{equation*}
    \mathrm tr P=\sum_{i=1}^n P_{ii}=0.
\end{equation*}
Let
\begin{equation*}
e= \frac1{\sqrt n} \mathbf{1} =  \frac1{\sqrt n}(1,\ldots,1)^T \in \mathbb R^{n \times 1}.
\end{equation*}
Using \eqref{Psumm}, we get
\begin{equation}\label{Pe}
    Pe=\frac{2}{\sqrt n}r.
\end{equation}
Consequently,
\begin{equation*}
    s:=e^TPe=\frac4n,
    \qquad
    t:=e^TP^2e=\|Pe\|^2=\frac{4R_2}{n}.
\end{equation*}

Decompose $P$ relative to the orthogonal splitting
\begin{equation*}
    \mathbb R^n= \mathrm{Span}(e)\oplus \mathrm{Span}(e)^\perp.
\end{equation*}
In this decomposition, write
\begin{equation*}
    P= Q
    \begin{pmatrix}
        s & b^T\\
        b & C
    \end{pmatrix} Q^T,
\end{equation*}
where $Q$ is an orthogonal matrix.
Then
\begin{equation*}
    \mathrm tr(P^2)=s^2+2\|b\|^2+\mathrm tr(C^2),
\end{equation*}
while
\begin{equation*}
    t=e^TP^2e=s^2+\|b\|^2.
\end{equation*}
Therefore
\begin{equation}
    F=\mathrm tr(P^2)=2t-s^2+\mathrm tr(C^2).
    \label{eq:F-block}
\end{equation}

The diagonal block $C$ has at most $3$ negative eigenvalues, by the Cauchy interlacing theorem (see \cite{HJ2012}), since $P$ has at most $3$ negative eigenvalues. Also
\begin{equation*}
    \mathrm tr C=\mathrm tr P-s=-s.
\end{equation*}

Thus the sum of the absolute values of the negative eigenvalues of $C$ is at least~$s$. Since there are at most $3$ such eigenvalues,
\begin{equation}\label{eq:trinity}
    \mathrm tr(C^2)\ge \frac{s^2}{3}.
\end{equation}

Using this in \eqref{eq:F-block}, we obtain
\begin{equation}\label{eq:lower-bound-preraw}
    F\ge 2t-s^2+\frac{s^2}{3}
    =2t-\frac{2s^2}{3}.
\end{equation}
Substituting $s=4/n$ and $t=4R_2/n$, this becomes
\begin{equation}
    F\ge \frac{8R_2}{n}-\frac{32}{3n^2}.
    \label{eq:lower-bound-raw}
\end{equation}

Finally, by Cauchy's inequality,
\begin{equation}\label{R2}
    R_2=\sum_{i=1}^n r_i^2
    \ge \frac{\left(\sum_i r_i\right)^2}{n}
    =\frac4n.
\end{equation}
Using $\alpha^2-4\alpha+8/3=0$, we rewrite the right-hand side of \eqref{eq:lower-bound-raw} as
\begin{equation}\label{fsum}
    \frac{8R_2}{n}-\frac{32}{3n^2}
    =\frac{8\alpha}{n}\left(R_2-\frac{\alpha}{n}\right)
    +\frac{8(1-\alpha)}{n}\left(R_2-\frac4n\right).
\end{equation}
Since $0<\alpha<1$ and $R_2\ge 4/n$, it follows that
\begin{equation}
    F\ge \frac{8\alpha}{n}\left(R_2-\frac{\alpha}{n}\right).
    \label{eq:lower-bound-final}
\end{equation}

The strict upper bound \eqref{eq:upper-bound} contradicts the lower bound \eqref{eq:lower-bound-final}. Therefore, $M$ cannot be entrywise strictly positive.

\quad

We now pass to the proof of Proposition \ref{pr}.

Note that the statement for $n = 3$ is trivial. Let us assume that the statement is true for $n-1$ rows and fix the matrix
\begin{equation*}
U = \left( \begin{matrix}
u_{11} & u_{12} \\
u_{21} & u_{22} \\
\dots & \dots \\
u_{n1} & u_{n2}
\end{matrix} \right) \in \mathbb{C}^{n \times 2}, \quad U^H U = I.
\end{equation*}

Similarly to \cite{SP2026}, we consider the following two possibilities.

\textbf{Case A.} At least one row of the matrix $U$ has a small norm. Specifically,
\begin{equation*}
\exists i: |u_{i1}|^2 + |u_{i2}|^2 \leq \frac{\alpha}{n}.
\end{equation*}
Without loss of generality, assume $i = 1$. We will show that in this case the corresponding inequality is not only true but also strict.

Introduce unitary matrix $Z \in \mathbb{C}^{2 \times 2}$ such that
\begin{equation*}
V := UZ = \left( \begin{matrix}
v & 0 \\
v_{21} & v_{22} \\
\dots & \dots \\
v_{n1} & v_{n2}
\end{matrix} \right), \text{ where } |v|^2 = |u_{11}|^2 + |u_{12}|^2 \leq \frac{\alpha}{n}.
\end{equation*}
Proving the statement for matrix $V$ is equivalent to proving it for $U$. The case $v = 0$ follows immediately from the inductive hypothesis, so let us concentrate on the case $v \neq 0$.

Put $t = 1/\sqrt{1 - |v|^2} > 1$ and define the matrix $\tilde{V} \in \mathbb C^{(n-1) \times 2}$, which has orthonormal columns, by
\begin{equation*}
\tilde{V} = \left( \begin{matrix}
t v_{21} & v_{22} \\
\dots & \dots \\
t v_{n1} & v_{n2}
\end{matrix} \right).
\end{equation*}

By the inductive hypothesis, there exists a $2 \times 2$ submatrix of $\tilde{V}$ with the spectral norm of its inverse less than or equal to $\sqrt{(n-1)/\alpha}$. Assume that
\begin{equation*}
\tilde{V}_{ij} = \left( \begin{matrix}
t v_{i1} & v_{i2} \\
t v_{j1} & v_{j2}
\end{matrix} \right) =
\left( \begin{matrix}
v_{i1} & v_{i2} \\
v_{j1} & v_{j2}
\end{matrix} \right)
\left( \begin{matrix}
t & 0 \\
0 & 1
\end{matrix} \right)
\end{equation*}
is such a submatrix.

Since $||\tilde{V}_{ij}^{-1}||_2 = \sigma_2^{-1}(\tilde{V}_{ij})$ and applying the min--max characterization of singular values (see, e.g., \cite{HJ2012}), we get
\begin{equation*}
\sigma_2\left( \begin{matrix}
v_{i1} & v_{i2} \\
v_{j1} & v_{j2}
\end{matrix} \right) \geq \sigma_2(\tilde{V}_{ij}) \, \sigma_2\left( \begin{matrix}
1/t & 0 \\
0 & 1
\end{matrix} \right) \geq \sqrt{\frac{\alpha}{n-1}} \times \frac{1}{t}.
\end{equation*}

This implies
\begin{align*}
\sigma_2^2\left( \begin{matrix}
v_{i1} & v_{i2} \\
v_{j1} & v_{j2}
\end{matrix} \right) \geq \frac{\alpha}{n-1} &\times \frac{1}{t^2} = \alpha \frac{1-|v|^2}{n-1} \geq \\ \geq \frac{\alpha(1 - \alpha/n)}{n-1} = &\frac{\alpha}{n} \frac{n-\alpha}{n-1} > \frac{\alpha}{n},
\end{align*}
which is equivalent to the strict version of the desired inequality.

Since the obtained inequality is strict, there is no need to analyze the equality case.

\textbf{Case B.} The norms of the rows of $U$ are large enough. Specifically\footnote{For distinction, $\langle\cdot,\cdot\rangle$ and $||\cdot||$ denote the standard Hermitian inner product and the corresponding norm on $\mathbb C^2$, whereas $(\cdot,\cdot)$ and $|\cdot|$ denote the standard Euclidean inner product and the corresponding norm on $\mathbb R^3$.},
\begin{equation}\label{assm}
||u_i||^2 = |u_{i1}|^2 + |u_{i2}|^2 > \frac{\alpha}{n} \text{ for all } i = 1, \ldots, n.
\end{equation}

Denote the rows of $U$ by $u_1, \ldots, u_n$ and their Hopf images by $w_1, \ldots, w_n \in~\mathbb{R}^3$.

By assumption \eqref{assm},
\begin{equation*}
|w_i| = ||u_i||^2 > \alpha/n.
\end{equation*}

Moreover, since $U^HU = I$, the following holds
\begin{equation}\label{ws}
\sum_{i = 1}^{n} w_i = 0, \quad \sum_{i = 1}^{n} | w_i | = 2.
\end{equation}

Also, similar to \cite{SP2026} we have
\begin{equation}\label{sqdot}
|\langle u_i, u_j\rangle|^2 = \frac{1}{2} | w_i | | w_j | + \frac{1}{2} (w_i, w_j), \text{ for all } i,j = 1, \ldots, n.
\end{equation}

Let us show that there exist distinct $i$ and $j$, such that
\begin{equation}\label{trgt}
|\langle u_i, u_j\rangle|^2 \leq (|| u_i ||^2 - \frac{\alpha}{n}) (|| u_j ||^2 - \frac{\alpha}{n}).
\end{equation}

Applying \eqref{sqdot} we get the equivalent inequality
\begin{equation*}
(w_i, w_j) - (| w_i | - \frac{2\alpha}{n}) (| w_j | - \frac{2\alpha}{n}) + \frac{2\alpha^2}{n^2} \leq 0.
\end{equation*}

By Lemma \ref{le}, inequality \eqref{trgt} is satisfied for some $i$ and $j$. Applying assumption \eqref{assm} to \eqref{trgt}, we conclude that $i \neq j$.

Let us consider the corresponding submatrix
\begin{equation*}
U_{ij} = \left( \begin{matrix}
u_{i1} & u_{i2} \\
u_{j1} & u_{j2}
\end{matrix} \right).
\end{equation*}

The characteristic polynomial of its Gram matrix $\Gamma_{ij} = U_{ij} U_{ij}^H$ is 
\begin{equation*}
\chi(\lambda) = (||u_i||^2 - \lambda) (||u_j||^2 - \lambda) - |\langle u_i, u_j\rangle|^2.
\end{equation*}

In Case B, $\mathrm{tr} \, \Gamma_{ij} = ||u_i||^2 + ||u_j||^2 > 2\alpha/n$. Hence $\lambda_1(\Gamma_{ij}) > \alpha/n$.

Since by \eqref{trgt} $\chi(\alpha/n) \geq 0$ and $\chi(\alpha/n) = (\lambda_1(\Gamma_{ij}) - \alpha/n)(\lambda_2(\Gamma_{ij}) - \alpha/n)$, we get
\begin{equation*}
\lambda_2(\Gamma_{ij}) \geq \frac{\alpha}{n}
\end{equation*}
and the desired inequality
\begin{equation*}
||U_{ij}^{-1}||_2 \leq \sqrt\frac{n}{\alpha}.
\end{equation*}

Let us now analyze the equality case. It consists in the fact that 
\begin{equation}\label{mx}
\max_{i \neq j} \lambda_2(\Gamma_{ij}) = \frac{\alpha}{n}.
\end{equation}

In terms of matrix $M$ from Lemma \ref{le}, \eqref{mx} means that $M$ is entrywise nonnegative and has at least one zero entry.

Repeating the argument leading to \eqref{eq:upper-bound} but with non-strict inequalities, gives
\begin{equation*}
F \leq \frac{8\alpha}{n}\left(R_2-\frac{\alpha}{n}\right).
\end{equation*}

Comparing this with \eqref{eq:lower-bound-final}, we obtain equality in \eqref{eq:lower-bound-final}.

Together with \eqref{eq:lower-bound-raw} and \eqref{fsum}, this gives $R_2 = 4/n$, which allows us to apply the equality criteria for the Cauchy's inequality \eqref{R2}. It gives
\begin{equation}\label{eqrs}
r_1 = \ldots = r_n = \frac{2}{n}.
\end{equation}

Thus, \eqref{Pe} takes the form
\begin{equation*}
P e = \frac{4}{n} e.
\end{equation*}
Hence, $\mathrm{Span}(e)$ and $\mathrm{Span}(e)^\perp$ are invariant subspaces of $P$.

Since in \eqref{eq:lower-bound-preraw} and \eqref{eq:trinity} the equalities also hold, matrix $C$ has exactly $3$ negative eigenvalues equal to $-s/3 = -4/(3n)$ (which implies that $n-1 \geq 3$), while the other eigenvalues are zeros.

Therefore, the spectrum of $P$ is
\begin{equation}\label{specP}
\frac{4}{n}, -\frac{4}{3n}, -\frac{4}{3n}, -\frac{4}{3n}, \, 0, \, \ldots, \, 0.
\end{equation}

Using \eqref{eqrs} and $\alpha^2-4\alpha+8/3=0$, rewrite
\begin{align*}
M &= -P + \frac{2\alpha}{n}(r\mathbf{1}^T + \mathbf{1}r^T) - \frac{2\alpha^2}{n^2} \mathbf{1}\mathbf{1}^T = \\
&= -P + \frac{16}{3n^2} \mathbf{1}\mathbf{1}^T = -P + \frac{16}{3n} ee^T.
\end{align*}

Together with \eqref{specP}, this determines the spectrum of $M$
\begin{equation*}
\frac{4}{3n}, \, \frac{4}{3n}, \, \frac{4}{3n}, \, \frac{4}{3n}, \, 0, \, \ldots, \, 0.
\end{equation*}

Since $M$ is real symmetric and entrywise nonnegative, after a simultaneous permutation of rows and columns it is block diagonal with irreducible nonnegative diagonal blocks. By the Perron-Frobenius theorem, each nonzero irreducible block has a simple positive dominant eigenvalue (see, for example, \cite{BP1994}). Since the spectrum of $M$ contains exactly four positive eigenvalues, all equal to $4/(3n)$, and since $M_{ii} > 0$ for every $i$, there are exactly four such blocks, each of rank one.

Therefore, $M = f_1 f_1^T + \ldots + f_4 f_4^T$ for certain nonnegative orthogonal vectors $f_1, \ldots, f_4 \in \mathbb{R}^{n \times 1}$. Since $\mathbf{1} \in \mathrm{Span}(f_1, \ldots, f_4)$, we can state that up to a simultaneous permutation of rows and columns
\begin{equation}\label{blck}
M = \frac{4}{3n} \times \mathrm{diag}(\frac{1}{n_1}E_{n_1}, \frac{1}{n_2}E_{n_2}, \frac{1}{n_3}E_{n_3}, \frac{1}{n_4}E_{n_4}),
\end{equation}
where $n_1 + \ldots + n_4 = n$ and $E_{n_1}, \ldots, E_{n_4}$ are the all-ones matrices of the corresponding sizes.

On the other hand, $M_{ii} = 16/(3n^2)$ for every $i$. Hence,
\begin{equation*}
n_1 = \ldots = n_4 = \frac{n}{4}.
\end{equation*}

Also, if $w_i$ and $w_j$ correspond to the same block (cluster)
\begin{equation*}
P_{ij} = |w_i||w_j| - (w_i, w_j) = 0,
\end{equation*}
which together with $r_i = r_j$ gives $w_i = w_j$. This yields
\begin{equation*}
|\langle u_i, u_j\rangle|^2 = |w_i||w_j| = ||u_i||^2||u_j||^2,
\end{equation*}
which means that $u_i$ and $u_j$ are collinear, and therefore $\lambda_2(\Gamma_{ij}) = 0$.

If $w_i$ and $w_j$ correspond to different blocks, we have
\begin{equation*}
P_{ij} = |w_i||w_j| - (w_i, w_j) = \frac{16}{3n^2},
\end{equation*}
which gives $\cos\angle(w_i, w_j) = -1/3$ and thus establishes the necessity of the formulated equality condition. 

Conversely, suppose that $4 \mid n$ and the Hopf images consist of four clusters of size $n/4$, each vector having length $2/n$, with inter-cluster angles satisfying $(w_i, w_j) = -4/(3n^2)$. Then for two rows of matrix $U$ corresponding to different clusters,
\begin{equation*}
|\langle u_i, u_j\rangle|^2 = \frac{1}{2}|w_i||w_j| + \frac{1}{2}(w_i, w_j) = \frac{4}{3n^2}.
\end{equation*}
Hence, the two eigenvalues of the corresponding Gram matrix $\Gamma_{ij}$ are
\begin{equation*}
\frac{2}{n} \pm \frac{2}{n\sqrt3},
\end{equation*}
and the smaller one is
\begin{equation*}
\frac{2-2/\sqrt3}{n} = \frac{\alpha}{n}.
\end{equation*}

In conclusion, consider the optimal constants for the discussed spectral norms
\begin{equation*}
b_n = \frac{1}{\sqrt{a_n}}, \text{ where } \, a_n = \inf_{U^HU=I_2} \max_{i \neq j} \lambda_2(U_{ij} U_{ij}^H).
\end{equation*}

Since $(b_n)_{n \geq 3}$ is nondecreasing and $b_{n} = \sqrt{n/\alpha}$ for $4 \mid n$,
\begin{equation}\label{conv}
\lim\limits_{n\rightarrow\infty} \frac{\sqrt{n/\alpha}}{b_n} = 1,
\end{equation}
which establishes the asymptotic tightness of $\sqrt{n/\alpha}$ and completes the proof of Proposition \ref{pr}.

\section{Acknowledgements}

I would like to thank Igor Makhlin, Stanislav Budzinskiy and the authors of \cite{GTZ1997} and \cite{SP2026} for fruitful discussions.

\bibliographystyle{plain}
\bibliography{lit}

\end{document}